\documentclass{amsart}

\usepackage{amsmath,amssymb,amsfonts,enumerate,amsthm,graphicx,color,hyperref}
\usepackage{tikz, float}

\renewcommand{\dim}{\mbox{dim}\,}
\renewcommand{\dim}{\mbox{dim}\,}

\newcommand{\set}{\mbox{set}\,}

\newcommand{\cochord}{\mbox{cochord}\,}

\newcommand{\T}{\mathrm}
\newcommand{\N}{\mathbb{N}}

\newtheorem{thm}{Theorem}[section]
\newtheorem{cor}[thm]{Corollary}
\newtheorem{lem}[thm]{Lemma}

\newtheorem{defn}[thm]{Definition}

\newtheorem{exam}[thm]{Example}

\numberwithin{equation}{section}
\tikzstyle{Cgray}=[draw=black, scale = .4,circle, fill = white, minimum size=7mm]
\tikzstyle{Cwhite}=[scale = .8,circle, fill = white, minimum size=8mm]
\tikzstyle{Cblack}=[scale = .3,circle, fill = black, minimum size=3mm]
\tikzstyle{C0}=[scale = .9,circle, fill = black!0, inner sep = 0pt, minimum size=3mm]
\tikzstyle{C1}=[scale = .7,circle, fill = black!0, inner sep = 0pt, minimum size=3mm]
\tikzstyle{CW}=[scale = .7,circle, fill = white!0, inner sep = 0pt, minimum size=7mm]
\tikzstyle{Cg}=[scale = .8,circle, fill = gray, minimum size=8mm]

\begin{document}
\bibliographystyle{amsplain}

\title[$t$-clique ideal and $t$-independence ideal of a graph]{$t$-clique ideal and $t$-independence ideal of a graph}
\author[S. Moradi]{Somayeh Moradi}
\address{Somayeh Moradi, Department of Mathematics, School of Science, Ilam University, P.O.Box 69315-516, Ilam, Iran.} \email{somayeh.moradi1@gmail.com}

\keywords{Cohen-Macaulay, t-clique ideal, linear resolution, shellable.\\
Email: somayeh.moradi1@gmail.com}
\subjclass[2010]{Primary 13D02, 13F55;    Secondary 16E05}

\begin{abstract}
\noindent
In this paper, we introduce and study families of squarefree monomial ideals called clique ideals and independence ideals that can be associated to a finite graph. A family of clique ideals with linear resolutions has been characterized. Moreover some families of graphs for which the quotient ring of their clique ideal is Cohen-Macaulay are introduced and some homological invariants of the clique ideal of a graph $G$ which is the complement of a path graph or a cycle graph, are obtained. Also some algebraic properties of the independence ideal of path graphs, cycle graphs and chordal graphs are studied.
\end{abstract}

\maketitle

\section*{Introduction}

Classifying all monomial ideals with some algebraic properties like having a linear resolution or being Cohen-Macaulay in general is not easy to deal with.
In this regard finding classes of monomial ideals with some special algebraic properties is important. In particular, finding a correspondence between some families of squarefree monomial ideals and some combinatorial objects such as graphs and simplicial complexes and characterizing the
algebraic invariants of the ideal in terms of the construction of the combinatorial object associated to it, has been studied extensively in the last few years. As the first sample of these ideals, squarefree monomial ideals of degree two had been considered as the edge ideals of simple graphs, which was first defined in \cite{V}. Later, some other squarefree monomial ideals attached to graphs like path ideals, generalized cover ideals, et cetera, have been studied and some new families of ideals with special algebraic properties had been characterized.

In this paper, we introduce and study the $t$-clique ideal and the $t$-independence ideal of a graph. The $t$-clique ideal is a natural generalization of the concept of the edge ideal of a graph. For a graph $G$, a complete subgraph of $G$ with $t$ vertices is called a $t$-clique of $G$. The ideal $K_t(G)$ generated by the monomials $x_{i_1}\cdots x_{i_t}$ of degree $t$ such that the induced subgraph of $G$ on the set $\{x_{i_1},\ldots, x_{i_t}\}$ is a complete graph, is called the $t$-clique ideal of $G$. Note that $K_2(G)=I(G)$.

The edge ideals of graphs with a linear resolution have been characterized in \cite{Fro} as follows.

\begin{thm}(\cite[Theorem 1]{Fro})\label{Fro1}
The graph $G$ is chordal if and only if $I(G^c)$ has a linear resolution.
\end{thm}
The paper is organized as follows. In the first section we give some preliminaries which are needed in the sequel.
In Section 2, we introduce the $t$-clique ideal of a graph and study this ideal for some classes of graphs. First we show that the `only if' part  of Theorem \ref{Fro1} holds for any nonzero $t$-clique ideal too, but the `if' part does not hold in general for $t$-clique ideals (see Corollary \ref{cor1} and Example \ref{exam1}).
Indeed in Theorem \ref{vs}, it is shown that for a chordal graph $G$, any nonzero $t$- clique ideal $K_t(G^c)$ is a vertex splittable ideal and hence has linear quotients and a $t$-linear resolution. Moreover, in Corollary \ref{cor2}, for a chordal graph $G$ an inductive formula for the graded Betti numbers and the projective dimension of $K_t(G^c)$ is presented.
Then in Theorem \ref{reg}, for an arbitrary graph $G$ an upper bound for the regularity of $R/K_t(G)$ is given in terms of some data from $G$.
As the main results of this section, it is shown that  if $G$ is the complement of a path graph or a cycle graph, then the Stanley-Reisner simplicial complex of $K_t(G)$ is pure shellable and hence $R/K_t(G)$ is a Cohen-Macaulay ring (see Theorems \ref{pathlinear} and \ref{cyclelinear}).

In Section 3, we consider the $t$-independence ideal of a graph $G$, which is defined as $$J_t(G)=\bigcap_{\{x_{i_1},\ldots,x_{i_t}\}\in \Delta_G}(x_{i_1},\ldots,x_{i_t}),$$ where $\Delta_G$ is the independence complex of $G$. Using the results obtained in Section 2, we show that for a chordal graph $G$, the Stanley-Reisner simplicial complex of  $J_t(G)$ is pure vertex decomposable and $R/J_t(G)$ is a Cohen-Macaulay ring and $\T{pd}(R/J_t(G))=t$ (see Theorem \ref{pdc}).

Also in Corollaries \ref{pathlinear2} and \ref{cyclelinear2} it is shown that $J_t(P_n)$ and $J_t(C_n)$ have linear resolutions if they are nonzero, where $P_n$ and $C_n$ are path graph and cycle graph with $n$ vertices, respectively. In Corollary \ref{cor3}  the graded Betti numbers of $J_t(P_n)$ are explained with a recursive formula. Finally it is shown that for $n\geq 2t$, $\T{pd}(R/J_t(C_n))=2t-1$, which depends only on $t$ (see Theorem \ref{pdcycle}).

\section{Preliminaries}

Throughout this paper, we assume that $G$ is a simple graph with the vertex set $V(G)=\{x_1, \dots, x_n\}$ and the edge set $E(G)$ and $R=k[x_1,\ldots,x_n]$ is a polynomial ring over a field $k$. For a simplicial complex $\Delta$, the set of \textbf{facets} (maximal faces) of $\Delta$ is denoted by $\mathcal{F}(\Delta)$. In this section, we recall some preliminaries which are needed in the sequel.

For a simplicial complex $\Delta$, and a face $F\in \Delta$, the \textbf{dimension} of $F$ is defined as $\dim(F)=|F|-1$ and $\dim(\Delta)=\max\{\dim(F):\ F\in \Delta\}$.

For a graph $G$, the\textbf{ independence complex} of $G$ is defined as follows.
$$\Delta_G=\{F\subseteq V(G):\ e\nsubseteq F, \forall e\in E(G)\}.$$
Any element of $\Delta_G$ is called an \textbf{independent set} of $G$.

For a simplicial complex $\Delta$ and $F\in \Delta$, the \textbf{link} of $F$ in
$\Delta$ is defined as $$\T{lk}_{\Delta}(F)=\{G\in \Delta: G\cap
F=\emptyset, G\cup F\in \Delta\},$$ and the \textbf{deletion} of $F$ is the
simplicial complex $$\T{del}_{\Delta}(F)=\{G\in \Delta: G\cap
F=\emptyset\}.$$

\begin{defn}\label{1.1}
{\rm A simplicial complex $\Delta$ is  called \textbf{vertex decomposable} if
$\Delta$ is a simplex, or $\Delta$ contains a vertex $x$ such that
\begin{itemize}
\item[(i)] both $\T{del}_{\Delta}(x)$ and $\T{lk}_{\Delta}(x)$ are vertex decomposable, and
\item[(ii)] every facet of $\T{del}_{\Delta}(x)$ is a facet of $\Delta$.
\end{itemize}
A vertex $x$ which satisfies condition (ii) is called a
\textbf{shedding vertex} of $\Delta$.}
\end{defn}


\begin{defn}
{\rm A simplicial complex $\Delta$ is called \textbf{shellable} if there exists an ordering $F_1<\cdots<F_m$ on the
facets of $\Delta$
such that for any $i<j$, there exists a vertex
$v\in F_j\setminus F_i$ and  $\ell<j$ with
$F_j\setminus F_\ell=\{v\}$. We call $F_1,\ldots,F_m$ a \textbf{shelling} for
$\Delta$.}
\end{defn}

For a monomial ideal $I$, the unique set of minimal generators of $I$ is denoted by $\mathcal{G}(I)$.
A vertex splittable ideal was defined in \cite{MK} as follows.

\begin{defn}
{\rm
A monomial ideal $I$ in $R=k[X]$ is called \textbf{vertex splittable} if it can be obtained by the following recursive procedure.
\begin{itemize}
\item[(i)] If $u$ is a monomial and $I=(u)$, $I=(0)$ or $I=R$, then $I$ is a vertex splittable ideal.
\item[(ii)] If there is a variable $x\in X$ and vertex splittable ideals $I_1$ and $I_2$ of $k[X\setminus \{x\}]$ so that $I=xI_1+I_2$, $I_2\subseteq I_1$ and $\mathcal{G}(I)$ is the disjoint union of $\mathcal{G}(xI_1)$ and $\mathcal{G}(I_2)$, then $I$ is a vertex splittable ideal.
\end{itemize}
With the above notations if $I=xI_1+I_2$ is a vertex splittable ideal, then $xI_1+I_2$ is called a \textbf{vertex splitting} for $I$ and $x$ is called a \textbf{splitting vertex} for $I$.
}
\end{defn}

\begin{defn}\label{1.2}
{\rm
A monomial ideal $I$ in $R=K[x_1,\ldots,x_n]$ has \textbf{linear quotients} if there exists an ordering $f_1, \dots, f_m$ on the minimal generators of $I$ such that the colon ideal $(f_1,\ldots,f_{i-1}):_R(f_i)$ is generated by a subset of $\{x_1,\ldots,x_n\}$ for all $2\leq i\leq m$. We show this ordering by $f_1<\dots <f_m$ and we call it \textbf{an order of linear quotients} for $I$.
Also for any $1\leq i\leq m$, $\set_I(f_i)$ is defined as
$$\set_I(f_i)=\{x_k:\ x_k\in (f_1,\ldots, f_{i-1}) :_R (f_i)\}.$$
}
\end{defn}

\begin{thm}\cite[Corollary 2.7]{Leila}\label{Leila}
Let $I$ be a monomial ideal with linear quotients with the ordering $f_1<\cdots<f_m$ on the minimal generators of $I$.
Then $$\beta_{i,j}(I)=\sum_{\deg(f_t)=j-i} {|\set_I(f_t)|\choose i}.$$
\end{thm}

Having linear quotients is a strong tool to determine some classes of ideals with linear resolution. The main result in this way is the following lemma.

\begin{lem}(See \cite[Lemma 1.5]{HT}.)\label{Faridi}
Let $I=(f_1, \dots, f_m)$ be a monomial ideal with linear quotients such that all the monomials $f_i$ are of the same degree. Then $I$ has a linear resolution.
\end{lem}

For a $\mathbb{Z}$-graded $R$-module $M$, the \textbf{Castelnuovo-Mumford regularity} (or briefly regularity)
of $M$ is defined as
$$\T{reg}(M) = \max\{j-i: \ \beta_{i,j}(M)\neq 0\},$$
and the \textbf{projective dimension} of $M$ is defined as
$$\T{pd}(M) = \max\{i:\ \beta_{i,j}(M)\neq 0 \ \text{for some}\ j\},$$
where $\beta_{i,j}(M)$ is the $(i,j)$th graded Betti number of $M$.

The notion of Betti splitting for monomial ideals was introduced in \cite{FHV} as follows.
\begin{defn} \cite[Definition 1.1]{FHV}
{\rm
Let $I$, $J$ and $K$ be monomial ideals in $R$ such that $\mathcal{G}(I)$ is the disjoint union of $\mathcal{G}(J)$ and $\mathcal{G}(K)$.
Then $I=J+K$ is a \textbf{Betti splitting} if $$\beta_{i,j}(I)=\beta_{i,j}(J)+\beta_{i,j}(K)+\beta_{i-1,j}(J\cap K),$$
for all $i\in \N$ and degrees $j$.}
\end{defn}

When $I=J+K$ is a Betti splitting, important homological invariants of $I$ are related to those invariants of the smaller ideals (see \cite[Corollary 2.2]{FHV}).


For a squarefree monomial ideal $I=( x_{11}\cdots
x_{1n_1},\ldots,x_{t1}\cdots x_{tn_t})$, the \textbf{Alexander dual ideal} of $I$, denoted by
$I^{\vee}$, is defined as
$$I^{\vee}:=(x_{11},\ldots, x_{1n_1})\cap \cdots \cap (x_{t1},\ldots, x_{tn_t}).$$
One can see that $(I^{\vee})^{\vee}=I$.

For a simplicial complex $\Delta$, the Stanley-Reisner ideal associated to $\Delta$ is denoted by $I_{\Delta}$ and for a squarefree monomial ideal $I$, Stanley-Reisner simplicial complex associated to $I$ is denoted by $\Delta_I$.

For a simplicial complex $\Delta$ with the vertex set $X=\{x_1, \dots, x_n\}$, the \textbf{Alexander dual simplicial complex} associated to $\Delta$ is defined as $\Delta^{\vee}=\{X\setminus F:\ F\notin \Delta\}.$

For a subset $C\subseteq X$, by $x^C$ we mean the monomial $\prod_{x_i\in C} x_i$ in the ring $k[x_1, \dots, x_n]$.
One can see that
$(I_{\Delta})^{\vee}=(x^{F^c} \ : \ F\in \mathcal{F}(\Delta)),$
where $F^c=X\setminus F$.
Moreover, one can see that  $(I_{\Delta})^{\vee}=I_{\Delta^{\vee}}$.

The following theorem which was proved in \cite{T}, relates the projective dimension and the regularity of a
squarefree monomial ideal to its Alexander dual.

\begin{thm}(See \cite[Theorem 2.1]{T}.) \label{1.3}
Let $I$ be a squarefree monomial ideal. Then
$\T{pd}(I^{\vee})=\T{reg}(R/I)$.
\end{thm}

Note that since $\T{pd}(R/I)=\T{pd}(I)+1$, $\T{reg}(R/I)=\T{reg}(I)-1$ and $(I^{\vee})^{\vee}=I$, the above theorem implies that $\T{pd}(R/I)=\T{reg}(I^{\vee})$.

For a simple graph $G$, the \textbf{edge ideal} of $G$ is defined as the ideal $I(G)=(x_ix_j:\ \{x_i,x_j\}\in E(G))$. It is easy to see that $I(G)$ can be viewed as the Stanley-Reisner ideal of the simplicial complex
$\Delta_{G}$ i.e., $I(G)=I_{\Delta_G}$.


A graph $G$ is called \textbf{chordal}, if it contains no induced cycle of length greater than or equal to $4$. Also $G$ is called \textbf{co-chordal} if the complement graph $G^c$ is a chordal graph.
For a vertex $v\in V(G)$, the set of neighbours of $v$ is denoted by $N_G(v)$ and we set $N_G[v]=N_G(v)\cup\{v\}$. A vertex $v$ of $G$ is called a \textbf{simplicial vertex} if the induced subgraph of $G$ on $N_G[v]$ is a complete
graph.

A path graph with $n$ vertices is denoted by $P_n$ and a cycle graph with $n$ vertices is denoted by $C_n$.


\section{The $t$-Clique ideal of a graph}

In this section we introduce the $t$-clique ideal of a graph and study algebraic properties of $t$-clique ideal of some families of graphs like co-chordal graphs and the complements of path graphs and cycle graphs.

\begin{defn}
Let $G$ be a graph. Any complete subgraph of $G$ with $t$ vertices is called a $t$-clique of $G$. The $t$-clique ideal of $G$ is an ideal of $R$ defined as
$$K_t(G)=(x_{i_1}\cdots x_{i_t}:\ G[\{x_{i_1},\ldots, x_{i_t}\}]\ \textrm{is a t-clique}).$$

Note that $K_2(G)=I(G)$.
\end{defn}

In the following it is shown that the $t$-clique ideal of a co-chordal graph is a vertex splittable ideal.

\begin{thm}\label{vs}
Let $G$ be a chordal graph. Then $K_t(G^c)$ is a vertex splittable ideal for any positive integer $t$.
\end{thm}

\begin{proof}
We prove the assertion by induction on $|V(G)|$. If $|V(G)|=1$, then for any integer $t>1$, one has $K_t(G^c)=0$ and $K_1(G^c)=(x)$, where $V(G)=\{x\}$ and the result is clear. Let $|V(G)|>1$ and inductively assume that for any chordal graph $G'$ with $|V(G')|<|V(G)|$, $K_t(G'^c)$ is vertex splittable for all positive integers $t$. Every chordal graph has a simplicial vertex due to Dirac \cite{Dirac}. Let $u$ be a simplicial vertex of $G$. So the induced subgraph of $G$ on the vertex set $N_G[u]$ is a complete graph. The minimal generators of $K_t(G^c)$ are the independent sets of $G$ of size $t$. Set $H=G\setminus N_G[u]$.
For any independent set $W$ of $G$ of size $t$, if $u\notin W$, then $x^W\in K_t((G\setminus u)^c)$ and if $u\in W$, then $N_G(u)\cap W=\emptyset$ and $W\setminus \{u\}\subseteq V(H)$ and $x^{W\setminus \{u\}}\in K_{t-1}(H^c)$. So
$K_t(G^c)\subseteq K_t((G\setminus u)^c)+uK_{t-1}(H^c)$. Also for any independent set $W'$ of $H$ of size $t-1$, $W'\cup\{u\}$ is an independent set of $G$ of size $t$, since $W'\cap N_G(u)=\emptyset$. So $ux^{W'}\in K_t(G^c)$. Also clearly $K_t((G\setminus u)^c)\subseteq K_t(G^c)$. So
$$K_t(G^c)=K_t((G\setminus u)^c)+uK_{t-1}(H^c).$$
Note that $G\setminus u$ and $H$ are both chordal graphs with fewer vertices than $G$. Thus by induction hypothesis $K_t((G\setminus u)^c)$ and $K_{t-1}(H^c)$ are vertex splittable ideals. Moreover, any minimal generator of $K_t((G\setminus u)^c)$ is of the form $x^F$, where $F\subseteq V(G)\setminus \{u\}$ is an independent set of $G\setminus u$ of size $t$. Thus
$|F\cap N_G(u)|\leq 1$, since $N_G(u)$ makes a clique. Therefore there exists some $w\in N_G(u)$ such that $F\setminus \{w\}$ is an independent set of $H$ of size at least $t-1$. So
$x^F\in K_{t-1}(H^c)$. Thus $K_t((G\setminus u)^c)\subseteq K_{t-1}(H^c)$ and $K_t((G\setminus u)^c)+uK_{t-1}(H^c)$ is a vertex splitting for $K_t(G^c)$ with the splitting vertex $u$.
\end{proof}

\begin{thm}\cite[Theorem 2.4]{MK}\label{vdlMK}
Any vertex splittable ideal has linear quotients.
\end{thm}

Now, using  Theorems \ref{vs} and \ref{vdlMK} and Lemma \ref{Faridi} we have the following corollary, which is a generalization of Theorem \ref{Fro1} in some sense.

\begin{cor}\label{cor1}
Let $G$ be a chordal graph such that $K_t(G^c)\neq 0$. Then $K_t(G^c)$ has linear quotients and hence a $t$-linear resolution for any positive integer $t$.
\end{cor}

\begin{thm}\cite[Theorem 2.8, Remark 2.10]{MK}\label{MKtheorem}
Let $I = xI_1 + I_2$ be a vertex splitting for the monomial ideal $I$.
Then $I = xI_1 + I_2$ is a Betti splitting. Moreover $$\beta_{i,j} (I) = \beta_{i,j-1}(I_1) + \beta_{i,j} (I_2) + \beta_{i-1,j-1}(I_2).$$
\end{thm}

Thus using Theorems \ref{vs} and \ref{MKtheorem} and Corollary \ref{cor1} we get the following corollary.
\begin{cor}  \label{cor2}
Let $G$ be a chordal graph and $u$ be a simplicial vertex of $G$. Set $I=K_t(G^c)$, $J=K_t((G\setminus u)^c)$ and $K=K_{t-1}((G\setminus N_G[u])^c)$. Then  $I=J+uK$ is a Betti splitting. Moreover, if $I\neq 0$, then

\begin{itemize}
  \item [(i)] $\beta_{i,j}(I)=\beta_{i,j}(J)+\beta_{i-1,j-1}(J)+\beta_{i,j-1}(K)$,
  \item [(ii)] If $J\ne 0$, then $\T{pd}(I)=\max\{\T{pd}(J)+1,\T{pd}(K)\}$,
  \item [(iii)] $\T{reg}(R/I)=t-1$.
\end{itemize}
\end{cor}

The converse of Corollary \ref{cor1} does not hold in general. In the following example, we give a non-chordal graph such that $K_t(G^c)$ has linear quotients and hence a linear resolution.

\begin{exam}\label{exam1}
Let $G$ be the graph depicted in the following figure. Then $K_3(G^c)=(x_1x_2x_3)$ has linear quotients and hence a $3$-linear resolution. But $G$ is not chordal.
\end{exam}

\begin{figure}[ht]\label{Fig1}
\centerline{ \centerline{\includegraphics[width=5cm]{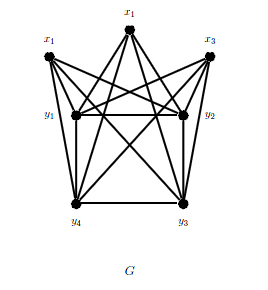}}}

\end{figure}

Let $c_t(G)$ be the minimum number of co-chordal subgraphs of $G$
required to cover the $t$-cliques of $G$, i.e., the minimum number of co-chordal subgraphs so that any $t$-clique of $G$ is contained in one of these subgraphs. Note that for $t=2$, $c_t(G)=\cochord(G)$ which was defined in \cite{W2}.
The following theorem, generalizes \cite[Theorem 1]{W2} with the similar proof.
\begin{thm}\label{reg}
Let $G$ be a graph. Then $\T{reg}(R/K_t(G))\leq (t-1) c_t(G)$.
\end{thm}

\begin{proof}
Let $H_1,\ldots,H_{c_t(G)}$ be co-chordal subgraphs of $G$ which cover the $t$-cliques of $G$. By Corollary \ref{cor2}, for any $1\leq i\leq c_t(G)$,  $\T{reg}(R/K_t(H_i))=t-1$ ( note that $K_t(H_i)\neq 0$, otherwise one can remove $H_i$ from the subgraphs and get $c_t(G)-1$ co-chordal subgraphs covering the $t$-cliques of $G$). We have $K_t(G)=\sum_{i=1}^{c_t(G)} K_t(H_i)$.
So $\T{reg}(R/K_t(G))\leq \sum_{i=1}^{c_t(G)} \T{reg}(R/K_t(H_i))=\sum_{i=1}^{c_t(G)}(t-1)=(t-1) c_t(G)$.
\end{proof}


In the sequel, we study $K_t(G)$, when $G$ is the complement of a path graph $P_n$ or a cycle graph $C_n$. We show that for these families of graphs $R/K_t(G)$ is a Cohen-Macaulay ring.
The following theorem shows that the Stanley-Reisner simplicial complex of $K_t(P_n^c)$ is pure shellable.
\begin{thm}\label{pathlinear}
Let $n$ and $t$ be positive integers. Then $\Delta_{K_t(P_n^c)}$ is pure shellable and hence $R/K_t(P_n^c)$ is Cohen-Macaulay. Also
$\dim(\Delta_{K_t(P_n^c)})=\left\{
\begin{array}{ll}
n-1  & \hbox{if}\ n<2t-1 \\
2t-3 & \hbox{if}\ n\geq 2t-1.
\end{array}
\right.$
\end{thm}

\begin{proof}
Let $P_n:x_1,\ldots,x_n$ be a path. Set $\Delta_{n,t}=\Delta_{K_t(P_n^c)}$.
For a subset $F\subseteq \{x_1,\ldots,x_n\}$, one has
 $$F\in \Delta_{n,t} \Longleftrightarrow   F\ \textrm{contains no independent set of}\ P_n\ \emph{of size t}.$$
Note that $P_n$  has an independent set of size $t$ if and only if $n\geq 2t-1$. In other words $K_t(P_n^c)\neq 0$ if and only if $n\geq 2t-1$. So if $n<2t-1$, then $K_t(P_n^c)=0$ and $\Delta_{K_t(P_n^c)}=\langle\{x_1,\ldots,x_n\}\rangle$ is a simplex. So it is pure shellable of dimension $n-1$.

So we assume that $n\geq 2t-1$. Any subset of $\{x_1,\ldots,x_n\}$ of size $2t-1$, contains an independent set of $P_n$ of size $t$. Indeed, for any set $\{x_{i_1},x_{i_2},x_{i_3},\ldots,x_{i_{2t-1}}\}$ of size $2t-1$, where $i_1<i_2<\cdots<i_{2t-1}$, $\{x_{i_1},x_{i_3},x_{i_5},\ldots,x_{i_{2t-1}}\}$ is an independent set of $P_n$ of size $t$.  Thus for any $F\in \mathcal{F}(\Delta_{n,t})$, one has $|F|\leq 2t-2$ and then $\dim(\Delta_{n,t})\leq 2t-3$.
If $t=1$, then clearly $\Delta_{n,t}=\{\emptyset\}$ and $\dim(\Delta_{n,t})=-1=2t-3$.
By induction on $n$, we show that for any positive integers $n$ and $t$ with $n\geq 2t-1$,  $\Delta_{n,t}$ is pure shellable of dimension $2t-3$. As was discussed above, we can assume that $t\geq 2$. For $n=3$, $\Delta_{3,2}=\langle \{x_1,x_2\},\{x_2,x_3\}\rangle$, which is pure of dimension $2t-3=1$ and $\{x_1,x_2\}<\{x_2,x_3\}$ is a shelling for it. Let $n>3$ and assume inductively that $\Delta_{m,t'}$ is pure shellable of dimension $2t'-3$ for any positive integers $m$ and $t'$ with $m\geq 2t'-1$.

Let $F\in \mathcal{F}(\Delta_{n,t})$. If $x_n\notin F$, then clearly $F\in \mathcal{F}(\Delta_{n-1,t})$. Note that $n-1\geq 2t-2$. If $n-1=2t-2$, then as was discussed in the first case $\Delta_{n-1,t}$ is pure shellable of dimension $n-2=2t-3$ and if $n-1\geq 2t-1$, then by induction hypothesis $\Delta_{n-1,t}$ is pure shellable of dimension $2t-3$.
Thus for any $F'\in \mathcal{F}(\Delta_{n-1,t})$, $|F'|=2t-2$ and then $F'\in \mathcal{F}(\Delta_{n,t})$ too. Therefore the facets of $\Delta_{n,t}$ which does not contain $x_n$, are precisely the facets of $\Delta_{n-1,t}$.

Now, let $F\in \mathcal{F}(\Delta_{n,t})$ and $x_n\in F$. We claim that $x_{n-1}\in F$. Otherwise $F\cup \{x_{n-1}\}\notin\Delta_{n,t}$ and then there is an independent set $S$ of $P_n$ of size $t$ such that  $S\subseteq F\cup \{x_{n-1}\}$ and $x_{n-1}\in S$. Clearly $S'=(S\setminus \{x_{n-1}\})\cup \{x_n\}$ is again an independent set of $P_n$ of size $t$, and $S'\subseteq F$, which is a contradiction. So the claim is proved and $\{x_{n-1},x_n\}\subseteq F$. Since $F$ contains no independent set of $P_n$ of size $t$, so $F'=F\setminus \{x_{n-1},x_n\}$ contains no independent set of $P_{n-2}$ of size $t-1$ (otherwise adding $x_n$ to this set, makes an independent set of $P_n$ of size $t$ contained in $F$). So $F'\in \Delta_{n-2,t-1}$.  Conversely, let $F'\in \mathcal{F}(\Delta_{n-2,t-1})$ be a facet. We show that $F'\cup\{x_{n-1},x_n\}\in \mathcal{F}(\Delta_{n,t})$. By contradiction let $S\subseteq F'\cup\{x_{n-1},x_n\}$ be an independent set of $P_n$ of size $t$. Then $|\{x_{n-1},x_n\}\cap S|=1$. Without loss of generality assume that $x_{n-1}\in S$ and $x_n\notin S$. Then $S\setminus \{x_{n-1}\}\subseteq F'$ is an independents set of $P_{n-2}:x_1,x_2,\ldots,x_{n-2}$ of size $t-1$, which is a contradiction. So $F'\cup\{x_{n-1},x_n\}$ is a face of $\Delta_{n,t}$. Since $n-2\geq 2(t-1)-1$, so by induction hypothesis $\Delta_{n-2,t-1}$ is pure shellable of dimension $2(t-1)-3=2t-5$. Thus $|F'\cup\{x_{n-1},x_n\}|=|F'|+2=2t-4+2=2t-2$. Thus $F'\cup\{x_{n-1},x_n\}\in \mathcal{F}(\Delta_{n,t})$.
Therefore the facets of $\Delta_{n,t}$ which contain $x_n$ are $G_1\cup\{x_{n-1},x_n\},\ldots,G_s\cup\{x_{n-1},x_n\}$, where $\Delta_{n-2,t-1}=\langle G_1,\ldots,G_s\rangle$.

Thus if $\Delta_{n-1,t}=\langle F_1,\ldots,F_r\rangle$, then  $$\Delta_{n,t}=\langle F_1,\ldots,F_r, G_1\cup\{x_{n-1},x_n\},\ldots,G_s\cup\{x_{n-1},x_n\}\rangle.$$

Since $|F_i|=|G_j\cup\{x_{n-1},x_n\}|=2t-2$ for any $1\leq i\leq r$ and $1\leq j\leq s$, so $\Delta_{n,t}$ is pure of dimension $2t-3$.
Assume inductively that $F_1<\cdots<F_r$ and $G_1<\cdots<G_s$ are shellings for $\Delta_{n-1,t}$ and $\Delta_{n-2,t-1}$, respectively.
We claim that $$F_1<\cdots<F_r<G_1\cup\{x_{n-1},x_n\}<\cdots<G_s\cup\{x_{n-1},x_n\}$$ is a shelling for $\Delta_{n,t}$.
To prove the claim it is enough to check the definition of shellability for two facets of the form $F_i$ and $G_j\cup\{x_{n-1},x_n\}$. One has $x_n\in G_j\cup\{x_{n-1},x_n\}\setminus F_i$. Moreover, $G_j\cup\{x_{n-1}\}$ contains no independent set of $P_{n-1}$ of size $t$, otherwise $G_j$ contains an independent set of $P_{n-2}$ of size $t-1$, which is a contradiction. Thus
$G_j\cup\{x_{n-1}\}\in \Delta_{n-1,t}$ and then $G_j\cup\{x_{n-1}\}\subseteq F_l$ for some $1\leq l\leq r$. So $G_j\cup\{x_{n-1},x_n\}\setminus F_l=\{x_n\}$.
\end{proof}

Now, we study the ideal $K_t(C_n^c)$, for the cycle graph $C_n$.

\begin{lem}\label{dimcycle}
Let $n$ and $t$ be positive integers. Then $\Delta_{K_t(C_n^c)}$ is a pure simplicial complex and
$\dim(\Delta_{K_t(C_n^c)})=\left\{
\begin{array}{ll}
n-1  & \hbox{if}\ n<2t \\
2t-3 & \hbox{if}\ n\geq 2t.
\end{array}
\right.$
\end{lem}

\begin{proof}
Let $C_n:x_1,\ldots,x_n$ be a cycle and set $\Delta'_{n,t}=\Delta_{K_t(C_n^c)}$.
For a subset $F\subseteq \{x_1,\ldots,x_n\}$, one has
$$F\in \Delta'_{n,t} \Longleftrightarrow   F\ \textrm{contains no independent set of}\ C_n\ \emph{of size t}.$$
Note that $C_n$  has an independent set of size $t$ if and only if $n\geq 2t$. So if $n<2t$, then $\Delta'_{n,t}=\langle\{x_1,\ldots,x_n\}\rangle$ is a simplex and hence pure of dimension $n-1$. So we assume that $n\geq 2t$. Let $F\in \mathcal{F}(\Delta'_{n,t})$. Since any subset of $\{x_1,\ldots,x_n\}$ of size $2t$ contains an independent set of $C_n$ of size $t$, we have $|F|\leq 2t-1<n$. Thus there exists $x_i\notin F$. Let $P$ be the induced subgraph of $C_n$ on the vertex set $\{x_1,\ldots,x_{i-1},x_{i+1},\ldots,x_n\}$ which is a path with $n-1$ vertices. Then $F$ is a face of $\Delta_{K_t(P^c)}$. So there exists a facet $F'\in \mathcal{F}(\Delta_{K_t(P^c)})$ such that $F\subseteq F'$. Note that $F'\in \Delta'_{n,t}$. Now, since $F\in \mathcal{F}(\Delta'_{n,t})$, we should have $F=F'$. Thus $F\in \mathcal{F}(\Delta_{K_t(P^c)})$. Now by Theorem \ref{pathlinear}, $\Delta_{K_t(P^c)}$ is pure of dimension $2t-3$. Thus $|F|=2t-2$. Therefore any facet of $\Delta'_{n,t}$ has dimension $2t-3$.
\end{proof}

\begin{thm}\label{cyclelinear}
Let $n$ and $t$ be positive integers. Then $\Delta_{K_t(C_n^c)}$ is pure shellable and hence $R/K_t(C_n^c)$ is Cohen-Macaulay.
\end{thm}

\begin{proof}
Let $C_n:x_1,\ldots,x_n$ be a cycle and $\Delta'_{n,t}=\Delta_{K_t(C_n^c)}$. If $n<2t$, then as was shown in the proof of Lemma \ref{dimcycle}, $\Delta'_{n,t}$ is a simplex. So it is pure shellable. We assume that $n\geq 2t$. Then again by Lemma \ref{dimcycle}, $\Delta'_{n,t}$ is  pure of dimension $2t-3$. We show that $\Delta'_{n,t}$ is shellable.

For any $1\leq i\leq n$, let $L_i$ be the induced subgraph of $C_n$ on the set $$\{x_1,\ldots,x_{i-1},x_{i+1},\ldots,x_n\}$$ which is a path graph and let $\Delta_i=\Delta_{K_t(L_i^c)}$ be a simplicial complex on the vertex set $\{x_1,\ldots,x_{i-1},x_{i+1},\ldots,x_n\}$. Then by Theorem \ref{pathlinear}, $\Delta_i$ is pure shellable of dimension $2t-3$ (note that $n-1\geq 2t-1$). Let $F\in \mathcal{F}(\Delta'_{n,t})$. Since $|F|=2t-2<n$, there exists $x_i\notin F$. Then $F$ is a facet of $\Delta_i$.  Also any facet of $\Delta_i$ is a facet of $\Delta'_{n,t}$, since $\Delta_i$ and $\Delta'_{n,t}$ are both pure of dimension $2t-3$. Thus $\Delta'_{n,t}=\Delta_1\cup \cdots \cup \Delta_n$. For any $1\leq i\leq n$, let $F_{i1}<\cdots<F_{is_i}$ be a shelling order on the facets of $\Delta_i$. Consider the ordering

\begin{equation}\label{shelling}
F_{11}<\cdots<F_{1s_1}<F_{21}<\cdots<F_{2s_2}<\cdots<F_{n1}<\cdots<F_{ns_n}.
\end{equation}

For any integers $i$ and $j$ with $i<j$, if $F_{ik}=F_{jr}$ for some $1\leq k\leq s_i$ and $1\leq r\leq s_j$, then we remove $F_{jr}$ from the above ordering. So we get an ordering $\mathcal{L}$ with no repeated terms on the facets of $\Delta'_{n,t}$.
We show that $\mathcal{L}$ is a shelling order for $\Delta'_{n,t}$.

Let $F_{ik},F_{jr}\in \mathcal{F}(\Delta'_{n,t})$ such that $F_{ik}<F_{jr}$ in $\mathcal{L}$. If $i=j$, then $F_{ik},F_{ir}\in \mathcal{F}(\Delta_i)$. So there exists $x_d\in F_{ir}\setminus F_{ik}$ and $\ell<r$ such that $F_{ir}\setminus F_{i\ell}=\{x_d\}$. Note that $F_{i\ell}<F_{ir}$ in $\mathcal{L}$.
Now, let $i\neq j$. Then $i<j$ and $x_i\notin F_{ik}$ and $x_j\notin F_{jr}$. We claim that
$x_i\in F_{jr}$. By contradiction assume that $x_i\notin F_{jr}$. Then $F_{jr}\in \Delta_i$. Since $|F_{jr}|=2t-2$ and $\Delta_i$ is pure of dimension $2t-3$, $F_{jr}$ is a facet of $\Delta_i$ too. So $F_{jr}=F_{ir'}$ for some $1\leq r'\leq s_i$. But it means that $F_{jr}$ has appeared twice in $\mathcal{L}$ which contradicts to the construction of $\mathcal{L}$. So $x_i\in F_{jr}$. Thus $F_{jr}\setminus \{x_i\}\in \Delta_i$ and so $F_{jr}\setminus \{x_i\}\subseteq F_{i\ell}$ for some $1\leq \ell\leq s_i$. Clearly $F_{i\ell}<F_{jr}$ in $\mathcal{L}$, $x_i\in F_{jr}\setminus F_{ik}$ and $F_{jr}\setminus F_{i\ell}=\{x_i\}$.
\end{proof}

\section{The $t$-independence ideal of a graph}

In this section, we consider the $t$-independence ideal of a graph and using its relation to the $t$-clique ideal and the results of the previous section, we obtain some homological invariants of the $t$-independence ideal of chordal graphs, path graphs and cycle graphs.

\begin{defn}
For a graph $G$, we define the $t$-independence ideal of $G$ as $$J_t(G)=\bigcap_{\{x_{i_1},\ldots,x_{i_t}\}\in \Delta_G}(x_{i_1},\ldots,x_{i_t}).$$ Indeed $J_t(G)=K_t(G^c)^{\vee}$.
\end{defn}

The following theorem was proved in \cite{MK}.

\begin{thm}\cite[Theorem 2.3]{MK}\label{vdMK}
Let $\Delta$ be a simplicial complex. Then $\Delta$ is vertex decomposable if and only if $I_{\Delta^{\vee}}$ is a vertex splittable ideal.
\end{thm}

Using Theorems \ref{vs} and \ref{vdMK}, we have the following Theorem.

\begin{thm}\label{pdc}
Let $G$ be a chordal graph. Then
\begin{itemize}
  \item [(i)] $\Delta_{J_t(G)}$ is pure vertex decomposable.
  \item [(ii)] $R/J_t(G)$ is  Cohen-Macaulay.
  \item [(iii)]  If $J_t(G)\neq 0$, then $\T{pd}(R/J_t(G))=t$.
  \item [(iv)] If $u$ is a simplicial vertex of $G$, then $J_t(G\setminus u)\neq 0$, then $\T{reg}(R/J_t(G))=\max\{\T{reg}(R/J_t(G\setminus u))+1,\T{reg}(R/J_{t-1}(G\setminus N_G[u]))\}.$
\end{itemize}
\end{thm}

\begin{proof}

$(i)$ By Theorems \ref{vs} and \ref{vdMK} $\Delta_{J_t(G)}$ is vertex decomposable. Note that $F\in \mathcal{F}(\Delta_{J_t(G)})$ if and only if $x^{F^c}$ is a minimal generator of $K_t(G^c)$. Since $K_t(G^c)$ is homogenous of degree $t$, thus $\Delta_{J_t(G)}$ is pure.

$(ii)$ By Corollary \ref{cor1}  $K_t(G^c)$ has a linear resolution. So by Eagon-Reiner Theorem $R/J_t(G)=R/K_t(G^c)^{\vee}$ is Cohen-Macaulay.

$(iii)$ Since $R/J_t(G)$ is  Cohen-Macaulay, by Auslander-Buchsbaum formula $$\T{pd}(R/J_t(G))=\dim(R)-\dim(R/J_t(G))=\T{ht}(J_t(G))=t.$$
$(iv)$ follows from Corollary \ref{cor2}(ii) and Theorem \ref{1.3}.
\end{proof}

We use the following theorem to prove Corollary \ref{pathlinear2}.

\begin{thm} (See  \cite[Theorem 1.4]{HD}.)\label{hdd}
A simplicial complex $\Delta$ is shellable if and only if $I_{\Delta^{\vee}}$ has linear quotients. Indeed $F_1<\cdots<F_m$ is  a shelling for $\Delta$ if and only if $x^{F_1^c}<\cdots<x^{F_m^c}$ is an order of linear quotients on the minimal generators of $I_{\Delta^{\vee}}$.
\end{thm}

Note that $K_t(P_n^c)\neq 0$ if and only if $n\geq 2t-1$.
Now, we get the following corollaries.

\begin{cor}\label{pathlinear2}
Let $n$ and $t$ be positive integers such that $n\geq 2t-1$. Then $J_t(P_n)$ has linear quotients and hence a $(n-2t+2)$-linear resolution.
\end{cor}

\begin{proof}
By Theorems \ref{pathlinear} and \ref{hdd}, $J_t(P_n)$ has linear quotients. Note that $x^F\in \mathcal{G}(J_t(P_n))$ if and only if $F^c=V(P_n)\setminus F\in \mathcal{F}(\Delta_{K_t(P_n^c)})$. By Theorem \ref{pathlinear} $\Delta_{K_t(P_n^c)}$ is pure of dimension $2t-3$, so any facet of it has cardinality $2t-2$ and then any minimal generator of
$J_t(P_n)$ has degree $n-(2t-2)$.  So by Lemma \ref{Faridi}, $J_t(P_n)$ has a linear resolution.
\end{proof}

Using above corollary  we can explain the projective dimension of the $t$-clique ideal (and hence the edge ideal) of the complement of a path graph.

\begin{cor}
Let $n$ and $t$ be positive integers such that $n\geq 2t-1$. Then $\T{pd}(K_t(P_n^c))=n-2t+1$. In particular, for $n\geq 3$, $\T{pd}(I(P_n^c))=n-3$.
\end{cor}

\begin{proof}
By Theorem \ref{1.3}, $\T{pd}(K_t(P_n^c))=\T{reg}(R/K_t(P_n^c)^{\vee})=\T{reg}(R/J_t(P_n))$. Now, by Corollary \ref{pathlinear2}, $\T{reg}(R/J_t(P_n))=\T{reg}(J_t(P_n))-1=n-2t+1$.
\end{proof}

The following theorem gives a recursive formula for the graded Betti numbers of the ideal $J_t(P_n)$.
\begin{thm}\label{cor3}
Let $n$ and $t$ be positive integers such that $n\geq 2t-1$. Then $$\beta_{i,j}(J_t(P_n))=\beta_{i,j-1}(J_t(P_{n-1}))+\beta_{i,j}(J_{t-1}(P_{n-2}))+\beta_{i-1,j-1}(J_{t-1}(P_{n-2})).$$
\end{thm}

\begin{proof}
Let $P_n:x_1,\ldots,x_n$ be a path. As was shown in the proof of Theorem \ref{pathlinear}, $\Delta_{n,t}=\Delta_{K_t(P_n^c)}$ is shellable and if $F_1<\cdots<F_r$ and $G_1<\cdots<G_s$ are shelling orders for  $\Delta_{n-1,t}$ and $\Delta_{n-2,t-1}$, respectively, then  $$F_1<\cdots<F_r< G_1\cup\{x_{n-1},x_n\},\cdots<G_s\cup\{x_{n-1},x_n\}$$ is a shelling order for $\Delta_{n,t}$. Set $X=\{x_1,\ldots,x_n\}$. By Theorem \ref{hdd},
$$x^{(X\setminus\{x_n\})\setminus F_1}<\cdots<x^{(X\setminus\{x_n\})\setminus F_r}$$ is an order of linear quotients for $J_t(P_{n-1})$,
$$x^{(X\setminus \{x_{n-1},x_n\})\setminus (G_1\cup\{x_{n-1},x_n\})}<\cdots<x^{(X\setminus \{x_{n-1},x_n\})\setminus (G_s\cup\{x_{n-1},x_n\})}$$
is an order of linear quotients for $J_{t-1}(P_{n-2})$ and
$$x^{X\setminus F_1}<\cdots<x^{X\setminus F_r}<x^{X\setminus (G_1\cup\{x_{n-1},x_n\})}<\cdots<x^{X\setminus (G_s\cup\{x_{n-1},x_n\})}$$
is an order of linear quotients for $J_t(P_n)$.
Set $f_t=x^{(X\setminus\{x_n\})\setminus F_t}$ for any $1\leq t\leq r$ and $g_t=x^{X\setminus(G_t\cup\{x_{n-1},x_n\})}$ for any $1\leq t\leq s$.
Then $$x_nf_1<\cdots<x_nf_r<g_1<\cdots<g_s$$ is an order of linear quotients for $J_t(P_n)$.
Also for any $1\leq t\leq r$, $$\set_{J_t(P_n)}(x_nf_t)=\set_{J_t(P_{n-1})}(f_t),$$ since $(x_nf_{\ell}):(x_nf_t)=(f_{\ell}):(f_t)$, and for any $1\leq t\leq s$, $$\set_{J_t(P_n)}(g_t)=\set_{J_{t-1}(P_{n-2})}(g_t)\cup\{x_n\},$$ since $(x_nf_{\ell}):(g_t)=(x_n)$ for some $1\leq \ell\leq r$.

So $$\beta_{i,j}(I)=\sum_{\deg(x_nf_t)=j-i} {|\set_{J_t(P_n)}(x_nf_t)|\choose i}+\sum_{\deg(g_t)=j-i} {|\set_{J_{t-1}(P_{n-2})}(g_t)|+1\choose i}=$$
$$\sum_{\deg(f_t)=j-i-1} {|\set_{J_t(P_{n-1})}(f_t)|\choose i}+\sum_{\deg(g_t)=j-i} {|\set_{J_{t-1}(P_{n-2})}(g_t)|\choose i}+\sum_{\deg(g_t)=j-i} {|\set_{J_{t-1}(P_{n-2})}(g_t)|\choose i-1}.$$

Thus $$\beta_{i,j}(J_t(P_n))=\beta_{i,j-1}(J_t(P_{n-1}))+\beta_{i,j}(J_{t-1}(P_{n-2}))+\beta_{i-1,j-1}(J_{t-1}(P_{n-2})).$$
\end{proof}

Finally, we consider the $t$-independence ideal of a cycle graph. Recall that $K_t(C_n^c)\neq 0$ if and only if $n\geq 2t$.

\begin{cor}\label{cyclelinear2}
Let $n$ and $t$ be positive integers such that $n\geq 2t$. Then $J_t(C_n)$ has linear quotients and hence a $(n-2t+2)$-linear resolution.
\end{cor}

\begin{proof}
By Theorems \ref{cyclelinear} and \ref{hdd}, $J_t(C_n)$ has linear quotients. Also $x^F\in \mathcal{G}(J_t(C_n))$ if and only if $F^c=V(C_n)\setminus F\in \mathcal{F}(\Delta_{K_t(C_n^c)})$. Since $\Delta_{K_t(C_n^c)}$ is pure of dimension $2t-3$ , so any minimal generator of
$J_t(C_n)$ has degree $n-(2t-2)$.  So by Lemma \ref{Faridi}, $J_t(P_n)$ has a linear resolution.
\end{proof}

\begin{cor}\label{cyclelinear3}
Let $n$ and $t$ be positive integers such that $n\geq 2t$. Then $\T{pd}(K_t(C_n^c))=n-2t+1$. In particular $\T{pd}(I(C_n^c))=n-3$.
\end{cor}

\begin{proof}
By Theorem \ref{1.3}, $\T{pd}(K_t(C_n^c))=\T{reg}(R/K_t(C_n^c)^{\vee})=\T{reg}(R/J_t(C_n))$. Now, by Corollary \ref{cyclelinear2}, $\T{reg}(R/J_t(C_n))=\T{reg}(J_t(C_n))-1=n-2t+1$.
\end{proof}

In the following theorem, it is shown that the projective dimension of $R/J_t(C_n)$ depends only on $t$.
\begin{thm}\label{pdcycle}
Let $n$  and $t$ be positive integers such that $n\geq 2t$. Then $$\T{pd}(R/J_t(C_n))=2t-1.$$
\end{thm}

\begin{proof}
Let $C_n:x_1,\ldots,x_n$ be a cycle. By Theorem \ref{1.3},
$\T{pd}(R/J_t(C_n))=\T{reg}(K_t(C_n^c))$.
Assume that $L$ is an induced subgraph of $C_n$ on the vertex set $V(C_n)\setminus \{x_1,x_{n-1},x_n\}$, which is a path. Let $W$ be an independent set of $C_n$ of size $t$. If $x_n\in W$, then $W\setminus \{x_n\}$ is an independent set of $L$ of size $t-1$.
If $x_n\notin W$, then $W$ is an independent set of $P_{n-1}$ too.
Conversely for any independent set $W'$ of $L$ of size $t-1$, $W'\cup\{x_n\}$ is an independent set of $C_n$ of size $t$. Indeed, we have $$K_t(C_n^c)=x_nK_{t-1}(L^c)+K_t(P_{n-1}^c).$$
Thus by \cite[Proposition 3.4]{MVi}, $$\T{reg}(K_t(C_n^c))\leq \T{reg}(x_nK_{t-1}(L^c))+\T{reg}(K_t(P_{n-1}^c))-1.$$
Note that $L$ and $P_{n-1}$ are chordal, so by Theorem \ref{pdc} and Theorem \ref{1.3}, $$\T{reg}(K_{t-1}(L^c))=\T{pd}(R/J_{t-1}(L))=t-1$$ and $\T{reg}(K_t(P_{n-1}^c))=\T{pd}(R/J_{t}(P_{n-1}))=t$. Also
$$\T{reg}(x_nK_{t-1}(L^c))=\T{reg}(K_{t-1}(L^c))+1=(t-1)+1=t.$$ Therefore
$$\T{pd}(R/J_t(C_n))=\T{reg}(K_t(C_n^c))\leq 2t-1.$$

We show that $\T{pd}(R/J_t(C_n))\geq 2t-1$.
For any $1\leq i\leq n$, let $L_i$ be the induced subgraph of $C_n$ on the set $\{x_1,\ldots,x_{i-1},x_{i+1},\ldots,x_n\}$ which is a path graph and let $\Delta_i=\Delta_{K_t(L_i^c)}$ be a simplicial complex on the vertex set $\{x_1,\ldots,x_{i-1},x_{i+1},\ldots,x_n\}$. Then by Theorem \ref{pathlinear}, $\Delta_i$ has a shelling order say $F_{i1}<\cdots<F_{is_i}$ on its facets. As was shown in the proof of Theorem \ref{cyclelinear}, $\Delta=\Delta_1\cup \cdots \cup \Delta_n$. So
any minimal generator of $J_t(C_n)$ is of the form $x^{F_{ik}^c}$ for some $1\leq i\leq n$ and $1\leq k\leq s_i$, where $F_{ik}^c=\{x_1,\ldots,x_n\}\setminus F_{ik}$.
Consider the shelling order $\mathcal{L}$ on $\Delta'_{n,t}=\Delta_{K_t(C_n^c)}$ as described in the proof of Theorem \ref{cyclelinear}. By Theorem \ref{hdd}, this induces an order of linear quotients on the minimal generators of $J_t(C_n)$. Also by Theorem \ref{Leila}, $$\T{pd}(J_t(C_n))=\max\{|\set_{J_t(C_n)}(u)|:\ u\ \textrm{is a minimal generator of}\ J_t(C_n)\}.$$
Set $F=\{x_1,x_2,\ldots,x_{2t-2}\}$.
Then $F$ contains no independent set of $C_n$ of size $t$ and $F\in \Delta_{2t-1}\setminus (\Delta_1\cup\cdots\cup \Delta_{2t-2})$. Since $|F|=2t-2$ and $\dim(\Delta_{2t-1})=2t-3$, so $F\in \mathcal{F}(\Delta_{2t-1})$.
We show that $\{x_1,x_2,\ldots,x_{2t-2}\}\subseteq \set_{J_t(C_n)}(x^{F^c})$. For any $1\leq i\leq 2t-2$, set $H_i=F\setminus \{x_i\}$.
The induced subgraph $C_n[H_i]$ is the union of (at most) two disjoint paths $x_1,x_2\ldots,x_{i-1}$ and $x_{i+1},x_{i+2},\ldots,x_{2t-2}$, so that one of them has odd number of vertices. One can see that if $i$ is an even number, then $H_i\cup \{x_n\}\in \mathcal{F}(\Delta_i)$ and if $i$ is an odd number, then $H_i\cup \{x_{2t-1}\}\in \mathcal{F}(\Delta_i)$.  Thus in the order of linear quotients for $J_t(C_n)$ induced by $\mathcal{L}$, we have  $x^{(H_i\cup \{x_n\})^c}<x^{F^c}$ for even $i$'s and $x^{(H_i\cup \{x_{2t-1}\})^c}<x^{F^c}$ for odd $i$'s. Moreover, $x^{(H_i\cup \{x_n\})^c}:x^{F^c}=x^{F\setminus (H_i\cup \{x_n\})}=x_i$ and $x^{(H_i\cup \{x_{2t-1}\})^c}:x^{F^c}=x^{F\setminus (H_i\cup \{x_{2t-1}\})}=x_i$. Thus for any $1\leq i\leq 2t-2$, $x_i\in \set_{J_t(C_n)}(x^{F^c})$. So $\T{pd}(J_t(C_n))\geq |\set_{J_t(C_n)}(x^{F^c})|\geq 2t-2$. Thus $$\T{pd}(R/J_t(C_n))=\T{pd}(J_t(C_n))+1\geq 2t-1.$$
\end{proof}


\providecommand{\bysame}{\leavevmode\hbox
to3em{\hrulefill}\thinspace}

\end{document}